\documentclass{article}
\usepackage{amssymb}
\usepackage{amsthm}
\usepackage{amsmath}
\usepackage{verbatim}
\usepackage[applemac]{inputenc}
\usepackage[english]{babel}

\newtheorem{theorem}{Theorem}[section]
\newtheorem{lemma}[theorem]{Lemma}
\newtheorem{e-proposition}[theorem]{Proposition}

\newtheorem{question}{Open question}
\newtheorem{defi}[theorem]{Definition\rm}
\theoremstyle{remark}
\newtheorem{remark}{Remark}

\newcommand{\ve}{\varepsilon}

\newcommand{\deb}{\rightharpoonup}

\newcommand{\pical}{\mathcal{P}}
\newcommand{\lcal}{\mathcal{L}}

\newcommand{\R}{\mathbb{R}}

\newcommand{\haus}{\mathcal{H}}

\newcommand\huno{{\mathcal H}^1}

\newcommand{\Ov}{\overline{\Omega}}
\newcommand\N{{\mathbb N}}

\newcommand{\MO}{\mathcal{M}(\Omega)}

\def\og{\leavevmode\raise.3ex\hbox{$\scriptscriptstyle\langle\!\langle$~}}
\def\fg{\leavevmode\raise.3ex\hbox{~$\!\scriptscriptstyle\,\rangle\!\rangle$}}

\title{A Modica-Mortola approximation for branched transport}
\author{Filippo Santambrogio \thanks{CEREMADE, Université Paris-Dauphine, 
Place de Lattre de Tassigny, F-75775 Paris Cedex 16, {\tt filippo@ceremade.dauphine.fr}}}

\begin{document}
\maketitle

\begin{abstract}
The $M^\alpha$ energy which is usually minimized in branched transport problems among singular 1-dimensional rectifiable vector measures with prescribed divergence is approximated (and convergence is proved) by means of a sequence of elliptic energies, defined on more regular vector fields. The procedure recalls the Modica-Mortola one for approximating the perimeter, and the double-well potential is replaced by a concave power.
\end{abstract}

\section{Introduction}

The name ``branched transport'' is now often used for addressing all the transport problems where the cost for a mass $m$ moving on a distance $l$ is proportional to $l$ but not to $m$ but sub-additive and typically proportional to a power $m^\alpha$ ($0<\alpha<1$). The distributions of sources and destinations are given and one looks for the path followed by each particle, sums up the mass which moves together on each part of the path, and associates to every configuration its total cost $\sum_i l_im_i^\alpha$. The adjective ``branched'' in the name stands for one of the main features of the optimal solutions: they gather mass together, masses tend to move jointly as long as possible, and then they branch towards different destinations, thus giving rise to a tree-shaped structure.

In the case of finite graphs this kind of problems dates back to the `60 in the community of operational research (the first paper on the subject is \cite{Gil}). More recently, several different approaches and results concerning the generalization to continuous frameworks have been introduced by the community of optimal transport. The first paper in this direction is the one by Q. Xia (\cite{xia1}), which will be our main reference. Other approaches have been proposed by Bernot, Caselles, Maddalena, Morel and Solimini (see \cite{MadMorSol, BerCasMor}). The equivalences between the different models as well as a survey on the whole theory are presented in a recent book by Bernot, Caselles and Morel,  \cite{Book Irrigation}, who are responsible for most of the results therein.

A satisfactory numerical treatment of these issues is far from being obtained and this is one of the reasons for the need for approximation results. This paper presents a precise approximation result, in terms of $\Gamma-$convergence (see \cite{introgammaconve}), which can be the origin of possible numerical methods based on PDEs, after an idea by Jean-Michel Morel. Actually, the continuous formulation of branched transport problems passes through a divergence-constrained formulation (see next session, where the model by Xia is quickly sketched) and it is natural to approximate it by means of problems which concern more regular vector fields (i.e. which are not measures concentrated on one-dimensional graphs, but have a density that is at least weakly differentiable).

We propose an approximation based on functionals with a concave term and a Dirichlet term, with coefficients which increase the weight of the first and vanish on the latter as far the approximation parameter $\ve$ goes to $0$. In this way one would need to solve a problem of the kind
\begin{equation}\label{Fve}
\min \int F(|u|) + \ve \int |\nabla u |^2
\end{equation}
under constraints or penalization on $\nabla\cdot u$. The optimality conditions for this problem read as an elliptic vector PDE (a system) with an unknown pressure coming from the divergence constraint. In dimension two, the decomposition of any vector field into a gradient plus an orthogonal gradient allows to translate this into a system of scalar fourth-order elliptic equations. Local minima may be found by a gradient descent method passing through a parabolic evolution equations, while global ones are more difficult to detect due to the concavity of $F$. Numerical methods for this kind of equations are much more developed and efficient rather than the combinatory-based ones which characterize the study of finite networks. A partial use of genetics algorithm could thereafter help in avoiding local minima.

Besides possibles numerical applications, the interest of this $\Gamma-$convergence result also comes from its comparison with the elliptic approximation of the perimeter functional, proposed by Modica and Mortola (\cite{ModMor}) at the beginning of $\Gamma-$convergence time. In their case $u$ was scalar and $F$ was a double-well potential, enforcing at the limit $u$ (after a suitable rescaling of the functional, so that the $F$-part has a coefficient going to infinity and the Dirichlet part a negligible one) to take values in $\{0,1\}$. And the energy was, in the limit as $\ve\to 0$, concentrated on a lower-dimensional structure, i.e. the interface between the two phases $u=0$ and $u=1$. The same will happen here: the concave power $F$ will play the role of a double well at $u=0$ and $|u|=\infty$  and the energy will concentrate on a one-dimensional graph.

Being $u$ a vector, in our problem, one could also evoke Ginzburg-Landau theory with its approximation (see \cite{BetBreHel 93, BetBreHel book} where the problem of the convergence for the minimizers of
$$
\min \int (1-|u|)^2 + \ve \int |\nabla u |^2
$$
is first addressed). Yet, it is easy to notice, due to the divergence which is a bounded measure, that the problem is essentially scalar, as at the limit there will be locally one direction only which will be relevant. Actually a singular vector measure, concentrated on a lower-dimensional object, must be oriented along a tangent direction if we want its divergence to be a measure (otherwise it is a first-order distribution).

Moreover, in dimension two it is possible to take advantage of the usual decomposition of a vector field into a gradient plus a rotated gradient so that information on the divergence ``fix'' the gradient part and all the functional may be expressed through the other gradient. In this way one would arrive to consider the limits of something like
$$
\min \int F(\nabla u)+ \ve \int |D^2 u |^2
$$
which are functionals of the form of those studied by Aviles and Giga (Modica-Mortola results for higher order energies, see \cite{AviGig87,AviGig89} where the second order term only contains the Laplacian, and lately \cite{AmbDeLMan} with the whole Hessian).

Probably the main goal of the paper is creating a a bridge between two different topics in Calculus of Variations: the approximation of free discontinuity problems on the one hand and the optimization of transport networks. The first one, much linked with elliptic PDEs has already been object of high-quality researches for decades and is studied in relation with its applications in material sciences and image segmentation. The second is, in its continuous version, more recent and linked to the theory of optimal transport by Monge and Kantorovitch, with applications ranging from economics to biology and geophysics (see \cite{RodRin}). Elliptic PDEs, dimensional reduction and geometric measure theory are very much involved as well. See \cite{AmbTor 90, AmbTor 92, Braides, ModMor} for the whole theory and the main examples of $\Gamma-$convergence applied to free discontinuity problems: notice that here, the discontinuities (or ``jumps'') are replaced by a bilateral singularity, in the sense that the rectifiable graphs in the limit problem vanish almost everywhere and are concentrated on one-dimensional sets, thus having a double jump, whose ``intensity'' (measured with respect to $\huno$ instead of $\lcal^d$) enters the limit functional.

The paper will start with two short sections on preliminaries, Section 2 on Xia's formulation of branched transport problems, Section 3 on $\Gamma-$convergence. Then, we will discuss briefly in Section 4 why to choose the form of the approximation we will choose, mainly concentrating on the choice of the exponents, since they are not obvious (in particular, the function $F$ in \eqref{Fve} will be of the form $F(|u|)=|u|^\beta$ with $\beta\neq \alpha$). Section 5 will present the detailed proof of the main result ($\Gamma-$convergence of the energies, in dimension two only, under no divergence conditions). Section 6 will suggest how to apply the result for producing interesting approximated variational problems, underlining what can be done with the tools we have so far and what could be worthwhile to prove in possible future investigations.

\section{Branched transport via divergence-constrained optimization}

We present here the framework of the optimization problem proposed by Xia in \cite{xia1, xia2} and then studied by many authors (see for instance \cite{Book Irrigation} for a whole presentation of the theory).

Let $\Omega\subset\R^d$ be an open set with compact closure $\Ov$ and $\mathcal{M}(\Omega)$ the set of finite vector measures on $\Ov$ with values in $\R^d$ and such that their divergence is a finite scalar measure, i.e. such that
\begin{equation}\label{norma div}
\sup \left\{\int \nabla\phi\cdot du\;:\; \phi\in C^1(\Ov), \;||\phi||_{L^\infty}\leq 1\right\}<+\infty
\end{equation}
(as you can see, we do not ask for $\phi$ vanishing at the boundary, i.e. we take into account possible parts of $\nabla\cdot u$ on $\partial \Omega$ as well). The value of the supremum in \eqref{norma div} will be denoted by $|\nabla\cdot u|(\Ov)$, i.e. the mass of the measure ``total variation of the divergence'' of $u$.
On this space we consider the convergence $u_\ve\to u \Leftrightarrow u_\ve\deb u \mbox{ and }\nabla\cdot u_\ve\deb \nabla\cdot u$ as measures. When a function is considered as an element of this space, or a functional space as a subset of it, we always think of absolutely continuous measures (with respect to the Lebesgue measure on $\Omega$) and the functions represent their densities.

When we take $u\in\MO$ and we write $u=U(M,\theta,\xi)$ we mean that $u$ is a rectifiable vector measure (it is the translation in the language of vector measures of the concept of rectifiable currents) $u=\theta\xi\cdot\huno_{|M}$ whose density with respect to the $\huno-$Hausdorff measure on $M$ is given by the real multiplicity $\theta: M\to\R^+$ times the orientation $\xi:M\to\R^d$, $\xi$ being a measurable vector field of unit vectors belonging to the (approximate) tangent space to $M$ at $\huno-$almost any point. 

For $0<\alpha<1$, we consider the energy

\begin{equation}\label{defi energy}
M^\alpha(u)=\begin{cases} \int_M \theta^\alpha d\huno &\mbox{ if } u = U(M,\theta,\xi),\\
                                                                +\infty &\mbox{ otherwise. } \end{cases}
                                                                \end{equation}
                                                                
The problem of branched transport amounts to minimizing $M^\alpha$ under a divergence constraint:
\begin{equation}\label{prob xia}
\min \left\{M^\alpha(u)\;:\;\nabla\cdot u = f := f^+-f^-\right\}.
\end{equation}
The divergence constraint is given in weak form and means 
$$\int \nabla\phi\cdot\,du= \int \phi\; d(f^-\!-\!f^+)\;\mbox{ for all }\phi\in C^0(\Ov),$$ 
which actually corresponds to Neumann boundary conditions
$$\nabla\cdot u = f \;\mbox{ in $\Omega$   and }\; u\cdot n = 0 \;\mbox{ on }\partial\Omega.$$
From now on, we will always think of Neumann boundary conditions when speaking about divergences, so that $\nabla\cdot u$ is the linear functional associating to every $\phi\in C^1(\Omega)$ (independently of the values on $\partial\Omega$) the value $\int \nabla\phi\cdot du$: if $u$ is a regular function this corresponds to a measure which absolutely continuous inside $\Omega$ with density given by the true divergence, and which has a $\huno-$part on the boundary with density given by $u\cdot n$.

\begin{remark}\label{xia relaxation}
Notice that this is not the original definition by Xia of the Energy $M^\alpha$: Xia proposed it in \cite{xia1} as a relaxation from the case of finite graphs, but formula \eqref{defi energy} can be seen as a representation formula for the relaxed energy
$$M^\alpha(u)=\inf\left\{\liminf_n E^\alpha(G_n)\;:\; G_n \mbox{ finite graph, }\;u_{G_n}\to u\right\},$$
where 
\begin{equation}\label{discrete problem}
E^\alpha(G):=\,\sum_h w_h^{\alpha}\haus^1(e_h),
\end{equation}
for a weighted oriented graph $G=(e_h,\, \hat{e}_h,\, w_h)_h$ (where $e_h$ are the edges, $\hat{e}_h$ their orientations, $w_h$ the weights), and $u_G$ is the associated vector measure given by 
$$u_G:=\sum_h w_h\hat{e}_h \huno_{|e_h},$$
(and the convergence is in the sense of $\MO$). For the proof of the equivalences between the two definition, look at \cite{xia2} or at Chapter 9 in \cite{Book Irrigation}.
\end{remark}

Notice that in general Problem \eqref{prob xia} admits a solution with finite energy for any pair of probability measures $(f^+,f^-)$ (or, more generally, for any pair of equal mass finite positive measures), provided $\alpha>1-1/d$ (this is proven in \cite{xia1} by means of an explicit construction).

\section{Variational approximation, preliminaries}\label{gamma conv sec}

The main result of the paper will be a $\Gamma-$convergence result for a sequence of energies approximating $M^\alpha$. We will see in Section 6 that for a complete approximation of the problem, one would need to insert the ``boundary conditions'' given by the divergence constraints and to prove compactness for a suitable sequence of minimizers $u_\ve$ of the approximating problems. 

For precising what we mean by ``approximating the energy'' and how to use the result, let us sketch briefly the main outlines of $\Gamma-$convergence's theory, as introduced by De Giorgi (see \cite{DeGFra} and \cite{introgammaconve}).

\begin{defi}
On a metric space $X$ let $F_n:X\to\R\cup\{+\infty\}$ be a sequence of functions. We define the the two lower-semicontinuous functions $F^-$ and $F^+$ (called $\Gamma-\liminf$and $\Gamma-\limsup$ $F^+$ of this sequence, respectively) by
\begin{gather*}
F^-(x):=\inf\{\liminf_{n\to\infty} F_n(x_n)\;:\;x_n\to x\},\\
 F^+(x):=\inf\{\limsup_{n\to\infty} F_n(x_n)\;:\;x_n\to x\}.
 \end{gather*}

Should $F^-$ and $F^+$ coincide, then we say that $F_n$ actually $\Gamma-$converges to the common value $F=F^-=F^+$. 
\end{defi}

This means that, when one wants to prove $\Gamma-$convergence of $F_n$ towards a given functional $F$, one has actually to prove two distinct facts: first we need $F^-\geq F$ (i.e. we need to prove $\liminf_n F_n(x_n)\geq F(x)$ for any approximating sequence $x_n\to x$; not only, it is sufficient to prove it when $F_n(x_n)$ is bounded) and then $F^+\leq F$ (i.e. we need to find a sequence $x_n\to x$ such that $\limsup_n F_n(x_n)\leq F(x)$). 

The definition of $\Gamma-$convergence for a continuous parameter $\ve\to 0$ obviously passes through the convergence to the same limit for any subsequence $\ve_n\to 0$.

Among the properties of $\Gamma-$convergence we have the following:
\begin{itemize}
\item if there exists a compact set $K\subset X$ such that $\inf_X F_n=\inf_K F_n$ for any $n$, then $F$ attains its minimum and $\inf F_n\to \min F$;
\item if $(x_n)_n$ is a sequence of minimizers for  $F_n$ admitting a subsequence converging to $x$, then $x$ minimizes $F$
\item if $F_n$ is a sequence $\Gamma-$converging to $F$, then $F_n+G$ will $\Gamma-$converge to $F+G$ for any continuous function $G:X\to\R\cup\{+\infty\}$.
\end{itemize}

\section{Elliptic approximation, intuition and heuristics}

As we partially mentioned, the result we will present in Section 5 is somehow inspired by, or at least recalls most of the results in the elliptic approximation of free discontinuity problems (Modica-Mortola, Ginzburg-Landau or Aviles-Giga). We will only mention the following (see \cite{ModMor} and \cite{Braides}) because of its simplicity, even if it is probably not the closest one in this two-dimensional setting where Aviles-Giga seems closer.

\begin{theorem}
Define the functional $F_\ve$ on $L^1(\Omega)$ through
$$F_\ve(u)=\begin{cases}\frac 1\ve \int W(u(x))dx + \ve \int |\nabla u(x)|^2dx&\mbox{ if }u\in H^1(\Omega);\\
+\infty&\mbox{ otherwise}.\end{cases}$$
Then, if $W(0)=W(1)=0$ and $W(t)>0$ for any $t\neq 0,1$, the functionals $F_\ve$ $\Gamma-$converge towards the functional $F$ given by
$$F(u)=\begin{cases}cPer(S)&\mbox{ if }u=1 \mbox{ on }S,\; u=0 \mbox{ on $S^c$ and $S$ is a finite-perimeter set};\\
+\infty&\mbox{ otherwise},\end{cases}$$
where the constant $c$ is given by $c=2\int_0^1\sqrt{W(t)}dt$.
\end{theorem}

We precised the value of the constant so that the reader will notice that similar constants are involved in our case as well. For the same reason (the analogy with the present paper) we precise also that one of the key-ingredient in the proof of the above Theorem is the inequality
$$\frac 1 \ve W(u(x))+ \ve |\nabla u (x)|^2\geq 2\sqrt{W(u(x)}||\nabla u (x)|=2|\nabla (H\circ u)|,$$
where $H$ is the primitive of $\sqrt{W}$ (so that one has $F_\ve(u)\geq 2TV(H\circ u)$).

In our study we will instead consider functionals of the form
\begin{equation}\label{Mve}
E_\ve(u)=\ve^{\gamma_1}\int_\Omega |u(x)|^\beta dx + \ve^{\gamma_2}\int_\Omega |\nabla u(x)|^2 dx,
\end{equation}
defined on $u\in H^1(\Omega;\R^2)$ and set to $+\infty$ outside $H^1\subset\MO$, for two exponents $\gamma_1<0<\gamma_2$. 

As one can see the functional recalls Modica-Mortola's functional to recover the perimeter as a limit, where the double-well potential is replaced with a concave power. Notice that concave powers, in their minimization, if the average value for $u$ is fixed in a region (which is in some sense the meaning of weak convergence, i.e. the convergence we use on $\MO$), prefer either $u=0$ or $|u|$ being as large as possible, i.e. there is sort of a double well on zero and infinity.

We give here a heuristics for determining the exponents $\beta$, $\gamma_1$ and $\gamma_2$. 

Suppose you want to approximate a measure $u$ concentrated on a segment $S$, with multiplicity $m$, and directed towards one of the direction of the segment, via a measure $u_A$whose density is smooth and concentrated on a strip of width $A$ around $S$ (for instance by convolution).

The values of $u_A$ will hence be of the order of $m/A^{d-1}$ and the values of its gradient of the order of $m/A^d$. This gives a functional of the order of
$$E_\ve\approx\ve^{\gamma_1}A^{d-1}\left(\frac{m}{A^{d-1}}\right)^\beta+\ve^{\gamma_2}A^{d-1}\left(\frac{m}{A^{d}}\right)^2.$$
In the minimization one will choose the optimal width $A$, depending on $m$ and $\ve$, and this gives 
$$A\approx \ve^{\frac{\gamma_2-\gamma_1}{2d-\beta(d-1)}}m^{\frac{2-\beta}{2d-\beta(d-1)}};\quad
E_\ve\approx  \ve^{\gamma_2-(\gamma_2-\gamma_1)\frac{d+1}{2d-\beta(d-1)}}m^{2-(2-\beta)\frac{d+1}{2d-\beta(d-1)}}.$$
The correct choice for a possible convergence result towards the energy \eqref{defi energy} which is proportional to $m^\alpha$ is obtained by imposing that the exponent of $m$ is $\alpha$ and the exponent of $\ve$ is zero, i.e.
$$\beta=\frac{2-2d+2\alpha d}{3-d+\alpha(d-1)};\quad \frac{\gamma_1}{\gamma_2}=\frac{(d-1)(\alpha-1)}{3-d+\alpha(d-1)}.$$
Notice that $\gamma_1$ and $\gamma_2$ may not both be determined since one can always replace $\ve$ with a power of $\ve$, thus changing the single exponents but not their ratio.
Notice also that the exponent $\beta$ is positive and less than $1$ as soon as $\alpha\in]1-1/d,1[$, which is the usual condition.

Finally, it is worthwhile to remark that for this choice of exponents the dependence of $A$ with respect to $\ve$ is of the form $A\approx \ve^{\gamma_2/(d+1)}$ (and $\gamma_2>0$). This implies $\lim_{\ve\to 0}A=0,$ which gives weak convergence of the approximation we chose (enlarging $u$ a strip of width $A$ without changing its mass) to $u$. It proves nothing, but it is coherent with the convergence result we want to prove.

We conclude by underlining the case of dimension $2$, since it will be the only one we will solve: in this case one has $\beta=\frac{4\alpha-2}{\alpha+1}$ and $\gamma_1/\gamma_2=(\alpha-1)/(\alpha+1)$.

\begin{remark}
From this heuristics and from the proof that we will present in the following, the reader may see that the construction only works for $\alpha>1-1/D$. This is quite astonishing if one thinks that, also for $\alpha\leq 1-1/d$, there are measures $f^\pm$ which admit possible solution with finite energy (in particular when both the measures are supported on a same lower-dimensional set). Yet, the problem lies in the kind of approximation we require, which uses measures $u_\ve$ which are more regular and in particular are absolutely continuous with respect to $\lcal^d$ and may not be concentrated on lower-dimensional sets.
\end{remark}

\section{Our main $\Gamma-$convergence result}

This Section will be devoted to the proof of the following theorem:

\begin{theorem}\label{main}
Suppose $d=2$ and $\alpha\in]1/2,1[$: then we have $\Gamma-$convergence of the functionals $M^\alpha_\ve$ to $cM^\alpha$, with respect to the convergence of $\MO$, as $\ve\to 0$, where $c$ is a finite and positive constant (the value of $c$ is actually $c=\alpha^{-1}\left(4c_0\alpha/(1-\alpha)\right)^{1-\alpha}$, being $c_0=\int_0^1\sqrt{t^\beta-t}dt$).
\end{theorem}

As usual in several $\Gamma-$convergence proofs, we will work separately on the two inequalities we need.

\subsection{$\Gamma-\liminf$ inequality}

In this part we will consider an arbitrary sequence $u_\ve\to u$ and we will suppose that $M^\alpha_\ve(u_\ve)$ is bounded. This implies at first that all the $u_\ve$ are $H^1$ functions, and that
\begin{equation}\label{Lp estimates}
\int |u_\ve|^\beta \leq C\ve^{1-\alpha}\quad\mbox{ and }\quad \int |\nabla u_\ve|^2\leq C\ve^{-1-\alpha}.
\end{equation}

{\bf Step 1}
{\it Lower bounds on $M^\alpha_\ve$}

Consider a rectangle $R\subset\Omega$.  Suppose for simplicity that it is oriented according to the $x$ and $y$ axes, i.e. $R=[a,b]\times [c,d]$, and set $v_\ve:=[(u_\ve)_x]_+$ (the positive part of the $x-$component of $u_\ve$) and $v_\ve':=\partial v_\ve/\partial y$. Consider also a function $\phi_1: R\to[0,1]$, depending on $y$ only, such that $\phi_1=1$ on $[a,b]\times [c+\delta, d-\delta]$ and $\phi_1=0$ for $y\in \{c,d\}$. For every $x$, set $R_x:=\{x\}\times  [c, d]$ and $R'_x:= \{x\}\times  [c+\delta, d-\delta]$.

Fix a value of $x\in[a,b]$ and let $A_\ve$ be the maximal value of $v_\ve$ on $R'_x$ (which is well-defined for a.e. $x$) and $L_\ve=A_\ve^{\beta-1}$, $f_\ve(t)=\sqrt{(t^\beta-L_\ve t)_+}$ and $F_\ve(t)=\int_0^t f_\ve(s)ds$. One can write
$$v_\ve^\beta=f_\ve^2(v_\ve)+L_\ve v_\ve-(v_\ve^\beta-L_\ve v_\ve)_-\geq f_\ve^2(v_\ve)+L_\ve v_\ve- L_\ve v_\ve I_{R_x\setminus R'_x}=f_\ve^2(v_\ve)+L_\ve v_\ve I_{R'_x}$$
where the second inequality comes from the fact that $v_\ve\leq A_\ve$ implies $v_\ve^\beta-L_\ve v_\ve\geq 0$ and $A_\ve=\max_{R'_x} v_\ve$.
Considering the other term as well one has
\begin{multline*}
\ve^{\alpha-1}v_\ve^\beta+\ve^{\alpha+1}(v_\ve')^2\geq \ve^{\alpha-1}L_\ve v_\ve I_{R'_x}+\ve^{\alpha-1}f_\ve^2(v_\ve)+\ve^{\alpha+1}(v_\ve')^2\\
\geq \ve^{\alpha-1}L_\ve v_\ve I_{R'_x}+2\ve^\alpha f_\ve(v_\ve)v_\ve'= \ve^{\alpha-1}L_\ve v_\ve I_{R'_x}+2\ve^\alpha (F_\ve(v_\ve))'.
\end{multline*}

By multiplying times $\phi_1$ and integrating on $R_x$ with respect to $y$, one has
\begin{multline}\label{with CS}
\int_{R_x}\left(\ve^{\alpha-1}v_\ve^\beta+\ve^{\alpha+1}(v_\ve')^2\right)\phi_1(y)dy\\
\geq \ve^{\alpha-1}L_\ve \int_{R_x'}v_\ve dy+2\ve^\alpha TV(F_\ve(v_\ve)\phi_1)- 2\ve^\alpha\int_{R_x\setminus R'_x}F_\ve(v_\ve)|\phi_1'|dy.
\end{multline}

Since $ v_\ve\phi_1$ vanishes at both boundaries and reaches the value $F_\ve(A_\ve)$ inside $R'_x$, the total variation $TV(F_\ve(v_\ve)\phi_1)$ is at least $2F_\ve(A_\ve)=2c_0A_\ve^{1+\beta/2}$. This value may be computed by a change of variable ($t=A_\ve s$):
$$F_\ve(A_\ve)=\int_0^{A_\ve}\sqrt{(t^\beta-L_\ve t)_+}dt=\int_0^1\sqrt{A_\ve^\beta s^\beta -L_\ve A_\ve s}\,A_\ve ds=A_\ve^{1+\beta/2}c_0$$
(we used $L_\ve=A_\ve^{\beta-1}$ and $t^\beta\geq L_\ve t$ for $t\leq A_\ve$). 

Notice that the last term in \eqref{with CS} will tend to zero as $\ve\to 0$, after integration with respect to $x$, thanks to Lemma \ref{sobolev} below, since $\phi_1$ is a fixed regular function and hence $\phi_1'$ is bounded, $F_\ve(t)\leq \int_0^t s^{\beta/2}ds=ct^{1+\beta/2}$ and $|v_\ve|\leq |u_\ve|$.

\begin{lemma}\label{sobolev}
For any bounded energy sequence $u_\ve$ we have 
$$\lim_{\ve\to 0}\ve^\alpha \int_R |u_\ve|^{1+\beta/2}=0.$$
\end{lemma}

\begin{proof}
First of all write $|u_\ve|\leq 1+w_\ve$, where $w_\ve= \left( |u_\ve|-1\right)_+$ and then 
$$\int_R |u_\ve|^{1+\beta/2}\leq C+ C\int_R w_\ve^{1+\beta/2}.$$
Notice than that $|\{|u_\ve|>1\}|\leq \int_R |u_\ve|^\beta \leq C\ve^{1-\alpha}\to 0$, thus $w_\ve$ vanishes on a large part of $R$. This allows to apply standard Sobolev-Poincaré inequalities $||w_\ve||_{L^r}\leq C||w_\ve||_{H^1}$ (in dimension two any exponent $r<+\infty$ is admitted). Remember $w_\ve\leq |u_\ve|$ and $|\nabla w_\ve|\leq |\nabla u_\ve|$.

Now, for any pair of conjugate exponents $p$ and $q$, we have
\begin{multline*}
\int_R w_\ve^{1+\beta/2} \leq\left(\int_\Omega |w_\ve|^\beta\right)^{\frac 1p}
\left(\int_\Omega |w_\ve|^{(1+\beta/2-\beta/p)q}\right)^{\frac 1q}\\
\leq C\ve^{(1-\alpha)/p}||u_\ve||_{L^{(1+\beta/2-\beta/p)q}}^{1+\beta/2-\beta/p}\leq C\ve^{(1-\alpha)/p}||u_\ve||_{H^1}^{1+\beta/2-\beta/p}\\
\leq C\ve^{(1-\alpha)/p-(1+\beta/2-\beta/p)(\alpha+1)/2}.
\end{multline*}
Hence we have
$$\ve^\alpha \int_R |u_\ve|^{1+\beta/2}\leq C\ve^\alpha+C\ve^{\gamma_p},$$ 
where the exponent $\gamma_p$, from the previous computations, is given by
$$\gamma_p=\alpha+\frac{1-\alpha}{p}-\frac{(2p+\beta(p-2))(\alpha+1)}{4p}=\frac{\alpha(2-p)}{2p}.$$
It goes to zero provided $\gamma_p>0$, and it is sufficient to choose $p<2$ in order to get the result.
\end{proof}
\begin{remark}
Notice that the proof would have been easier if one supposed that $u_\ve\cdot n=0$ (since one could have directly applied Sobolev-Poincaré to $u_\ve$) on $\partial\Omega$, which is quite natural. Yet, for the sake of generality, we admitted possible divergences concentrated on the boundary, i.e.non vanishing values of the normal component. Actually, in this $H^1$ setting, the divergence of $u_\ve$ is seen as a measure on $\Ov$ which belongs to $L^2(\Omega)+L^2(\partial\Omega)$.
\end{remark}

Hence, we will ignore the last term and, for every $x$, we look at the quantity
$$\ve^{\alpha-1}A_\ve^{\beta-1} m'_\ve(x) +4c_0\ve^\alpha A_\ve^{1+\beta/2},$$
where $m'_\ve(x):= \int_{R_x'}v_\ve dy$. To estimate it from below, we will minimize over possible values of $A_\ve$. We have
\begin{equation}\label{min A ve}
\min_{A\in ]0,+\infty[}  \ve^{\alpha-1}A^{\beta-1} m+4c_0\ve^\alpha A^{1+\beta/2} = c_2m^\alpha,
\end{equation}
the minimum being realized by 
\begin{equation}\label{optimal A}
A=\left(\frac{m(1-\beta)}{\ve 2c_0 (2+\beta)}\right)^{2/(4-\beta)}.
\end{equation}
 We do not precise here the coefficient $c_2$ appearing in the minimal value but the correct computation is the one in the statement of the theorem. Notice that the exponent $2/(4-\beta)$ equals $(\alpha+1)/3$. For computing the last equality in \eqref{min A ve} we need to use the relations between $\beta$ and $\alpha$.

This is the first point where we start seeing an expression recalling somehow $M^\alpha$.

{\bf Step 2}
{\it Comparison with $M^\alpha$}
 
 One can call $\mu_\ve$ the positive measure $\left(\ve^{\alpha-1}|u_\ve|^\beta + \ve^{\alpha+1} |\nabla u_\ve|^2 \right)\cdot\lcal^2$. Since $\int_\Omega d\mu_\ve = M^\alpha_\ve(u_\ve)$, the measures $\mu_\ve$ stay bounded in the set of positive Radon measures on $\Omega$. Hence we can suppose $\mu_\ve\deb\mu$.
 
As we did previously, take a rectangle $R$ and keep the same notations.  Keeping the positive  $x-$part only and the derivative with respect to $y$ only one has  $\mu_\ve\geq \left(\ve^{\alpha-1}v_\ve^\beta + \ve^{\alpha+1} |v'_\ve|^2 \right)\cdot\lcal^2$, hence we get 
\begin{equation}\label{with m'}
\int_R \phi_1(y)d\mu_\ve\geq \int_a^b [m'_\ve(x)]^\alpha dx-r_\ve,
\end{equation}
$r_\ve$ being a negligible rest corresponding to the last term in \eqref{with CS}.

  We would like this estimate to pass to the limit as $\ve\to 0$, using $u_\ve\to u$. This is not yet possible: to get it, take again a function $\phi_2:R\to[0,1]$, similar to $\phi_1$, depending on $y$ only, but such that $\phi_2=1$ on $[a,b]\times [c+2\delta, d-2\delta]$ and $\phi_2(y)=0$ for $y\in [c,c+2\delta]\cup[d-2\delta,d]$. For every $x$, set 
$$m''_{\ve,+}(x):= \int_{R_x}v_\ve\phi_2 dy,\; m''_{\ve,-}(x):= \int_{R_x}[(u_\ve)_x]_-\phi_2 dy,\;m_\ve(x):= \int_{R_x}(u_\ve)_x\phi_2 dy.$$
Notice that estimate \eqref{with m'} implies $\int_R \phi_1(y)d\mu_\ve\geq \int_a^b [m''_{\ve,+}(x)]^\alpha dx$ (up to the negligible term $r_\ve$). Moreover, performing for every $x$ the same proof as before with the negative part instead of the positive one, and using $\max \{ m''_{\ve,+}(x), m''_{\ve,-}(x)\}\geq |m_\ve(x)|$, one gets
\begin{equation}\label{with |m|}
\int_R \phi_1(y)d\mu_\ve\geq \int_a^b |m_\ve(x)|^\alpha dx.
\end{equation}

The measures $m_\ve(x)dx$ are the projections on $[a,b]$ of $\phi_2\cdot(u_\ve)_x$. They obviously weakly converge to the projection of $\phi_2 u_x$. Yet, this weak convergence is not sufficient for getting the convergence (nor for lower semicontinuity) of the right hand side of \eqref{with |m|}. This is due to the non-convex behavior of the function $m\mapsto |m|^\alpha$. To prove this convergence we need more compactness (and hence a stronger convergence of $(\pi_x)_\#(\phi_2 (u_\ve)_x)$).

What we may prove is that the functions $x\mapsto m_\ve(x)$ are uniformly $BV$ in $x$ and this will allow for $L^1$ and pointwise convergence. We will use the fact that $u_\ve\to u$ in $\MO$ implies a bound on $\nabla\cdot u_\ve$. Take a function $\psi:[a,b]\to \R$ with $\psi(a)=\psi(b)=0$. Consider
\begin{multline*}
\int_a^b m_\ve(x)\psi'(x)dx=\int_R (u_\ve)_x(x,y) \phi_2(y)\psi'(x)\,dydx\\
=\int_R u_\ve(x,y)  \cdot \nabla(\phi_2(y)\psi(x))dydx-\int_R (u_\ve)_y(x,y)  \phi_2'(y)\psi(x)\,dydx\\
\leq |\nabla\cdot u_\ve|(\Ov) ||\psi||_{L^\infty(a,b)} + ||u_\ve||_{L^1}  ||\phi_2'||_{L^\infty(a,b)} ||\psi||_{L^\infty(a,b)} \leq C ||\psi||_{L^\infty(a,b)}.
\end{multline*}

This proves the BV bound we needed and implies that $(\pi_x)_\#(\phi u_x)$ is a measure on $[a,b]$ which ``belongs to BV'' (i.e. is absolutely continuous and has a BV density). Let us call $m(x)$ its density.  Moreover, one has $m_\ve(x)\to m(x)$ for almost any $x$. Passing to the limit in \eqref{with |m|} as $\ve\to 0$ one gets, by Fatou's Lemma,
$$\int_R \phi_2(y)d\mu\geq \int_a^b |m(x)|^\alpha dx. $$

It is quite straightforward that one can let $\delta$ go to $0$ and get rid of the functions $\phi_1$ and $\phi_2$ (which actually depend on $\delta$). This gives
\begin{equation}\label{already established}
\mu([a,b]\times ]c,d[)\geq  \int_a^b |\tilde{m}(x)|^\alpha dx,
\end{equation}
where $\tilde{m}(x)$ is defined as the density at $x$ of the projection $\pi_x$ of the measure $u_x$ restrained at $[a,b]\times ]c,d[$. Notice that this measure has a density as one can realize by writing it as a series of measures with densities with finite series of $L^1$ norms (i.e. we take a sequence of $\delta_n$ going to zero and see that the corresponding $m_n$ are all $L^1$ functions, and the series $\sum_n ||m_{n+1}-m_{n}||_{L^1}$ converges since $\sum_n \int |(\phi_2)_{n+1}-(\phi_2)_{n}|d|u_x|<+\infty$).

It is useful to see that this estimate implies that $u$, simply because of its attainability as a limit of $u_\ve$ with bounded $M^\alpha_\ve$ energies, is a rectifiable measure. This is proven thanks to the following Lemmas \ref{rectifiability}, \ref{Galpha}.

\begin{lemma}\label{rectifiability}
Suppose $u_\ve\to u$ in $\MO$ and $M^\alpha_\ve(u_\ve)\leq C$. Then $u$ is a one-dimensional rectifiable vector measure, i.e. it is of the form $u=U(M,\theta,\xi)$.
\end{lemma}
\begin{proof}
We will use the rectifiability theorem proved through different techniques by Federer and White (see \cite{federer, white}) and already used to prove rectifiability of measures with finite $M^\alpha-$mass (see also \cite{xia2}). This theorem roughly states that $u$ is rectifiable if and only if almost all its $(d-1)-$dimensional slices parallel to the coordinate axis are countable collections of Dirac masses. 
Take the coordinate axis and disintegrate $u_x$ with respect to its projection $m(x)dx$ on the variable $x$, thus getting signed measures $\nu_x$. For every $n$, divide $\Omega$ into $2^n$ horizontal strips of equal width and call $f_{i,n}(x)$ the integral of $\nu_x$ on the $i$-th interval strip (without its boundary). The estimate \eqref{already established} that we already established easily gives
$$\mu(\Omega)\geq \int  m(x)^\alpha \sum_i f_{i,n}(x)^\alpha dx.$$
Up to choosing the levels where to put the boundary of the strips, we can ensure that for almost any $x$ no mass is given to the boundaries by $|u|$ (if this is not the case, simply translate a little bit the strips, and this will happen again for a countable set of choices only). Hence, when we pass from $n$ to $n+1$, the mass of the previous strip is exactly the sum of the mass of the two new strips.
Notice that, due to $|a+b|^\alpha\leq |a|^\alpha+|b|^\alpha,$ the sequence $n\mapsto  \sum_i f_{i,n}(x)^\alpha$ is increasing. Call $G(x)$ its limit: we have $\int m(x)^\alpha G(x) dx<+\infty$. 
This implies that $G$ is finite almost everywhere. Thanks to Lemma \ref{Galpha}, almost every $\nu_x$ is purely atomic, and the same may be performed on the direction $y$. This allows to apply the white's criterion and proves that $u$ is rectifiable.
\end{proof}

\begin{lemma}\label{Galpha}
For $\alpha<1$ and a measure $\nu$ on a interval (say $[0,1[$), set
$$G^{(sup)}_\alpha(\nu)=\sup\left\{\sum_{i=0}^{2^n-1} \left|\nu\left(I_{i,n}\right)\right|^\alpha;\quad n\in\N\quad I_{i,n}=\left[\frac{i}{2^n},\frac{i+1}{2^n}\right[\right\}$$
and
$$G_\alpha(\nu)=\begin{cases}
                 \sum_{k\in\N}(a_k)^\alpha & \textrm{if} \ \nu=\sum_{k\in\N}a_k\delta_{x_k} \\
                 +\infty & \textrm{otherwise} 
                 \end{cases}.$$
Then we have $ G^{(sup)}_\alpha(\nu)= G_\alpha(\nu)$, and in particular if $G^{(sup)}_\alpha(\nu)<+\infty$ then $\nu$ is purely atomic.
\end{lemma}
\begin{proof}
Let us start from proving $G^{(sup)}_\alpha(\nu)\geq G_\alpha(\nu)$: for any $n$, build a measure $\nu_n$ which is purely atomic, with one atom at each point $i2^{-n}$ and choose the mass of such an atom equal to that of $\nu$ on $I_{i,n}$. This measures converge weakly to $\nu$, hence we have
$$G_\alpha(\nu)\leq\liminf G_\alpha(\nu_n)\leq G^{(sup)}_\alpha(\nu).$$

Then, we prove the opposite inequality. We can suppose $G_\alpha(\nu)<+\infty$. Hence $\nu=\sum_{k\in\N}a_k\delta_{x_k}$ is purely atomic and for every pair $(i,n)$ we have, thanks to subadditivity,
$$\left|\nu\left(I_{i,n}\right)\right|^\alpha\leq \sum_{k\,:\,x_k\in I_{i,n}}a_k^\alpha.$$
Summing up, we get
$$\sum_{i=0}^{2^n-1} \left|\nu\left(I_{i,n}\right)\right|^\alpha\leq G_\alpha(\nu)$$
and the proof is obtained by taking the sup over $n$.
\end{proof}

 Once we know about $u$ being rectifiable, one can choose rectangles $R$ shrinking around a tangent segment to the set $M$ at a point $x_0$ (this works for $\huno-$almost any point  $x_0$) and get $\mu\geq \theta^\alpha\cdot\huno_{|M}$, which implies the thesis.

This kind of proofs follow a standard scheme, i.e. considering measures $\mu_\ve$ whose mass gives the value of the approximating energy, and providing estimates on the  limit measure $\mu$. This estimates are obtained through local inequality on $\int_R d\mu$, so that shrinking $R$ around a point one gets information on the density of $\mu$ (here it is the density w.r.t. $\huno$). Similar proofs (typically comparing $\mu$ to $\lcal^d$ instead of $\huno$), are quite used in $\Gamma-$convergence problems in the setting of transport and location: see for instance \cite{BouJimRaj} and \cite{MosTil}

\subsection{$\Gamma-\limsup$ inequality}

{\bf Step 3}
{\it The case of a single segment}

Consider the case $u=\theta\cdot\huno_{|S}$, being $S$ a segment (for simplicity, $S=[0,1]\times \{0\}$). For producing a ``recovery sequence'' one can inspire himself at the lower bound proof. We will look for a profile $u_\ve$ with the following properties:
\begin{itemize}
\item the $x-$component only of $u_\ve$ must be present and must have the same sign, so that $u_\ve=v_\ve e_1$;
\item $\nabla v_\ve=v'_\ve e_2$ (i.e. $v_\ve$ only depends on $y$);
\item the Cauchy-Schwartz inequality used in \eqref{with CS} must be an equality, i.e. one needs 
$$ v'_\ve = \pm \frac 1\ve \sqrt{v_\ve^\beta-L_\ve v_\ve };$$
\item the total variation must actually be given by twice the maximum (i.e. $v_\ve$ must be monotone on the two separate intervals before and after reaching the maximum: we will realize it by taking a maximal value at $y=0$ and symmetric monotone profiles around $0$);
\item $v_\ve$ must vanish at the boundary of a certain rectangle, so that one gan avoid using the function $\phi_1$;
\item the maximum $A_\ve$ must be optimal in \eqref{min A ve};
\item as $\ve\to 0$ weak convergence to the measure $u=\theta\cdot\huno_{|S}$ is needed: we will realize it by taking different rescaling of the same profile $z$ (say, $v_\ve(y)=A_\ve z(A_\ve y)$), so that $A_\ve\to \infty$ and $\int_\R z(t)dt=\theta$ will be sufficient.
\end{itemize}
Look at the conditions that the profile $z$ must satisfy: we need $z(0)=1$ and $z'(t)t\leq 0,$ so that the maximum of $v_\ve$ will be $A_\ve,$ realized at zero, and the monotonicity conditions as well will be satisfied. Moreover, on $t\geq 0$ (for $t\leq 0$ just symmetrize), $z$ must satisfy
$$z'(t)=-\frac{A_\ve^{\beta/2}}{\ve A_\ve^2}\sqrt{z^\beta-z}.$$
Thanks to \eqref{optimal A}, the ratio $A_\ve^{\beta/2}/(\ve A_\ve^2)$ equals $(1-\beta)\theta/(2c_0(2+\beta)).$ This is very good since there is no more $\ve$.

Notice that this equation has a strong non-uniqueness, since the function $z\mapsto \sqrt{z^\beta-z}$ is non-Lipschitz. For instance the constant one is a solution, but a solution going from $1$ to $0$ exists as well (just get it by starting from a different starting point and see that it has to reach both $0$ and $1$ in finite time).  This means that there are several solutions, and $\int_0^\infty z(t)dt$ may be any possible value larger than $\int_0^\infty z_0(t)dt$ ($z_0$ being the only solution with no flat part $z=1$ around $t=0$).

We want now to check that $\theta$ is larger than the lower bound for the integrals, and this gives the final property that $z$ needed to satisfy.

We need
$$\frac \theta 2 \geq \int_0^\infty z_0(t)dt= \int_0^1 \frac{z\,(1-\beta)\, \theta}{\sqrt{z^\beta-z}\,(2+\beta)\,2c_0}dz$$
(the integral has been computed by change of variables $z=z_0(t)$, which is possible for the solution $z_0$ only, since it is injective).
This means that we must compare $C_0:=\int_0^1 \frac{z}{\sqrt{z^\beta-z}}dz$ and $c_0=\int_0^1 \sqrt{z^\beta-z}dz$ and prove $C_0\leq c_0(2+\beta)/(1-\beta)$.

Compute $c_0$ by integrating by part:
$$c_0=\int_0^1 1\cdot \sqrt{z^\beta-z}dz= - \int_0^1 z\frac{\beta z^{\beta-1}-1}{2\sqrt{z^\beta-z}}dz=\frac{(1-\beta)}{2} C_0-\frac\beta 2 c_0,$$
which implies $c_0=C_0 \frac{1-\beta}{2-\beta}$ and is enough to get the desired inequality.

Obviously, one needs after that to perform a correction of $u_\ve$ near $x=0$ and $x=1$ so that the function actually belongs to $H^1$ (in order to avoid discontinuities at the two ends of the segment), and to control the extra energy one pays, as well as the divergence. 

One possibility for the case of the single segment is the following: the profile we got for $v_\ve$ is of the form $A_\ve z(A_\ve t)$ and one can simply replicate it radially on a half disk, so as to ensure regularity. In this case we need to estimate. Denote for simplicity by $f_\ve(r)$ the radial profile one performs and by $B_\ve$ the half ball where it is non-null.
We need to estimate
$$\ve^{\alpha-1}\int_{B_\ve}|f_\ve|^\beta + \ve^{\alpha+1}\int_{B_\ve}|f'_\ve|^2,\quad\int_ {B_\ve}|f'_\ve| ,\mbox{ and }\int_ {B_\ve}|f_\ve| $$
Just use $|f_\ve|\leq A_\ve,$ $|f_\ve|\leq C A_\ve^2$ and $|B_\ve|\leq C A_\ve^{-2}$ (and the value of $A_\ve$ which is of the order of $\ve^{-2/(4-\beta)}$): both terms in the energy will be of the order of $\ve^{(\alpha+1)/3}$, the second term to be estimated (the divergence) will be bounded and the last will be of the order of $\ve^{(\alpha+1)/3}$ as well.

In this way we have produced a sequence $u_\ve$ that converges to $u$ (weak convergence is ensured by construction, weak convergence of $\nabla\cdot u_\ve$ comes from the bound we just proved).

{\bf Step 4}
{\it Conclusion by density}        

As we pointed out in Section 2, it is well known from the works by Q. Xia on (see \cite{xia1}) that the energy $M^\alpha$ is obtained as a relaxation of the same energy defined on finite graphs.
This implies that the class of finite graphs is ``dense in energy'' in the space $\MO$ and $\Gamma-$convergence theory guarantees that it is enough to build recovery sequences for such a class (see \cite{introgammaconve}). One can also impose the condition $G\cap\partial\Omega=\emptyset$, so to avoid problems at the boundary.

For dealing with $u=u_G$, where $G$ is a finite graph, one can simply consider separately the segments composing $G$ and apply the previous construction of the previous step. Possible superpositions of the part of $u_\ve$ coming from different segments will happen only on regions whose size is of the order of $A_\ve^{-1}$ and hence negligible in the limit (and on such a region, one can use $|u_1+u_2|^\beta\leq |u_2|^\beta+|u_2|^\beta$ and $|\nabla(u_1+u_2)|^2\leq 2|\nabla u_1|^2+2|\nabla u_1|^2$). Not only, the number of nodes will be finite and hence the bound on the divergence will stay valid.

{\bf Improvement}
{\it Better connections at the junctions}

The construction for the $\Gamma-\limsup$ that we just detailed provides a sequence $u_\ve$ converging to $u$ in $\MO$ which works very well in the case of a single segment but which could be improved in general. Actually, in the case of a single segment the divergence of $u_\ve$ has the same mass as that of $u$, since it replaces two Dirac masses at the two extremal points with two diffuse masses, concentrated on half-balls of radius $cA_\ve^{-1}$ around the points. The mass of $u_\ve$ itself is only slightly larger of the mass of $u$, due to the part we added at the extremities. 

This changes a lot when one considers more than one segment, since the divergence of $u_\ve$ at the nodes will be given by the superposition of different densities, each corresponding to a segment: the integral will vanish due to compensations, but $||\nabla\cdot u_\ve||_{L^1}$ could have increased a lot.

This may be corrected thanks to the following lemma.

\begin{lemma}\label{lorenzo}
If $g\in C^1(B_R)$ is a function with zero mean on a ball of radius $R$, then there exists a vector field in $C^1(\overline{B_R})$ such that $\nabla\cdot v = g$, $v=0$ on $\partial B_R$ and
$$||v||_{L^\infty}\leq C\left(R||g||_{L^\infty}+R^2||\nabla g||_{L^\infty}\right);\quad ||\nabla v||_{L^2}^2\leq C\left(R^2||g||_{L^\infty}^2+R^4||\nabla g||_{L^\infty}^2\right),$$
where $C$ is a universal constant.
\end{lemma}
\begin{proof}
Take $w$ the solution of the elliptic problem
$$\begin{cases}\Delta w =g&\mbox{ in }B_R,\\
			 w=0 &\mbox{ on }\partial B_R.\end{cases}$$
Thanks to $\int g = 0$ one has $\int_{\partial B_R} \partial w/\partial n =0$. This allows to define a function $\phi:\partial B_R\to \R$ with $\phi'=\partial u/\partial n$. Take then a cut-off function $\chi(r)$ such that $\chi(r)=1$ if $r\in[2R/3, R]$ and $\chi(r)=0$ if $r\in [0,R/3]$, $|\chi|\leq 1,$ $|\chi'|\leq C/R$, $|\chi''|\leq C/R^2$ and define $\psi(x)=\chi(|x|)\phi(Rx/|x|)$.

Now take $v=\nabla w +Rot( \nabla \psi)$, where $Rot$ denotes a 90° clockwise rotation.  In this way $\nabla\cdot v = \nabla \cdot \nabla w = g, $ since the rotated gradient part is divergence-free, and both the normal and the tangential component of $v$ on the boundary vanish (since the tangential component of $\nabla \psi$ compensates the normal one of $\nabla $ and the tangential component of $\nabla w$ and the normal of $\nabla\psi$ are zero). 

We only need to check the bounds on the norms. These bounds come from standard elliptic regularity theory (see for instance \cite{gilbargtrudinger}), since one has
\begin{eqnarray*}
||\nabla w||_{L^\infty}&\leq& C\left(R||g||_{L^\infty}+R^2||\nabla g||_{L^\infty}\right),\\
||\phi||_{L^\infty}&\leq& CR||\nabla\phi||_{L^\infty}\leq CR||\nabla w||_{L^\infty},\\
||D^2w||_{L^2}^2&\leq& C||\Delta w||_{L^2}^2=C||g||_{L^2}^2,\\
||D^2\phi||_{L^2}^2&\leq& CR||D^2 w||_{L^2(\partial B_R)}^2\leq CR^2||D^3w||_{L^2}^2+C||D^2w||_{L^2}^2\\
&\leq& CR^2||\nabla g||_{L^2}^2+C||g||_{L^2}^2.
\end{eqnarray*}
The last line of inequalities come from the combination of a trace inequality in Sobolev spaces applied to $D^2w$ (where the two coefficients of the gradient and the function part have different scaling with respect to $R$) with a regularity estimate for Dirichlet problems (estimating the $H^{k+2}$ norm of the solution with the $H^k$ norm of the datum).
Combining al the ingredients give the desired estimate (we pass to the $L^\infty$ norms for the sake of simplicity).
\end{proof}

With this lemma in mind, we can notice that at every node, the divergence of the function $u_\ve$ we gave before is composed by the zero-mean sum of some functions $g_i$ of the form $f'(r)(x\cdot e_i)/r$ (this is the divergence of the vector field directed as $e_i$ with radial intensity we used above), where $e_i$ is the direction of the corresponding segment. Each function is supported on a half ball whose simmetry axes follows $e_i$ and the radius is of the order of $A_\ve^{-1}\approx \ve^{2/(4-\beta)}$. Setting for each node $g=\sum_i g_i$ one has $\int g=0$, the support of $g$ is included in $B_R$ with $R\approx \ve^{2/(4-\beta)}$ and $||g||_{L^\infty}\leq CR^{-2}$ and $||\nabla g||_{L^\infty}\leq CR^{-3}$.

Hence, one can add at every node a vector field $v$ as in Lemma \ref{lorenzo} so as to erase the extra divergence. This would cost no more than something of the order of
$$\ve^{\alpha-1}R^2||v||_{L^\infty}^\beta+\ve^{\alpha+1}||\nabla v||_{L^2}^2\leq C\ve^{\alpha-1}R^{2-\beta}+C\ve^{\alpha+1}\frac{1}{R^2}=C\ve^{(\alpha+1)/3}.$$

Also the mass of the vector field is not that changed, since we only added $R^2||v||_{L^\infty}\leq C\ve^{(\alpha+1)/3}$.

\section{Applications and perspectives}

Besides the interesting comparison aspects of this result with respect to the similar ones in the approximation of free discontinuity problems, one of the main goal of this study concerned possible numerical applications, as we mentioned in the introduction. 

We want to replace the problem of minimizing $M^\alpha$ under divergence constraints with a simpler problem, i.e. minimizing $M^\alpha_\ve$. 

The idea would be solving
\begin{equation}\label{prob xia ve}
\min \left\{M^\alpha_\ve(u)\; : \; \nabla\cdot u = f_\ve\right\},
\end{equation}
being $f_\ve$ a suitable approximation of $f=f^+-f^-$., and proving that the minimizers of \eqref{prob xia ve} converge to the minimizers of \eqref{prob xia}. 

Theorem \ref{main} proves a $\Gamma-$convergence result, which should give the convergence of the minimizers, but the problem is that we never addressed the condition $\nabla\cdot u = f_\ve$, nor we discussed the choice of $f_\ve$. We will come back to this question, which is still open, later.

What we will consider here are penalization methods, which are quite natural as far as numerics is concerned. One could decide to replace the (quite severe) condition $\nabla\cdot u = f$ at the limit with a weaker one, concerning a distance between $\nabla\cdot u$ and $f$, possibly with a very high penalization coefficient. 
Considering problems of the form
\begin{equation}\label{prob xia ve pen}
\min \quad M^\alpha_\ve(u)+G( \nabla\!\cdot\! u;  f),
\end{equation}
where $G$ is a functional defined on pairs of finite measures, which is continuous w.r.t. weak convergence, and such that $G((\mu,\nu),(f^+,f^-))=0$ if and only if $\mu=\nu$, would play the game, since the $\Gamma-$convergence result could be easily extended, thanks to the last property of this kind of convergence that we presented in Section \ref{gamma conv sec}.

Yet, this would converge to 
\begin{equation}\label{prob xia pen}
\min \quad M^\alpha(u)+G(\nabla\!\cdot\! u;  f),
\end{equation}
 which in general is not exactly the same as imposing $\nabla\cdot u = f$. If on the contrary one wants a penalization of the kind $\ve^{-1}G$, so that at the limit one needs $G=0$, i.e. $\nabla\cdot u = f$, then the whole $\Gamma-$convergence result would have to be re-established (and would present some difficulties), since we are no more in the same framework of adding a same continuous functional.

Hence, we will stick to the case of a single penalization and we will start from the following consideration which works quite well if the divergence is given by the difference of two fixed mass positive measures.

As explained in \cite{xia1}, the quantity $d_\alpha(\mu,\nu):=\min \{M^\alpha(u)\; : \; \nabla\!\cdot\! u = \mu-\nu\}$ is a distance on the set of positive measures with the same finite mass that metrizes weak convergence.

As a consequence of $d_\alpha$ being a distance, due to triangular inequality, solving
$$\min_{\mu^+,\mu^-\in\pical(\Ov)}2d_\alpha(f^+,\mu^+)+\min \{M^\alpha(u)\; : \; \nabla\cdot u = \mu^+-\mu^-\}+2d_\alpha(f^-,\mu^-)$$
amount to choosing $\mu^\pm=f^\pm$ and solving $\min \{M^\alpha(u)\; : \; \nabla\cdot u = f^+-f^-\}$.

This means that choosing $2d_\alpha$ as a penalization is a clever strategy (even if it only allows for considering probabilities or, in general, fixed mass measures).

To describe a precise problem where we can insert this penalization we need to introduce a slightly different space. Consider the space $Y(\Omega)\subset\MO\times\pical(\Ov)\times\pical(\Ov)$ defined by
$$Y(\Omega):=\{(u,\mu,\nu)\in\MO\times\pical(\Ov)\times\pical(\Ov)\;:\;\nabla\cdot u =\mu-\nu\},$$
endowed with the obvious topology of componentwise weak convergence.

On this space consider the functionals 
\begin{equation}\label{functionals with pen}
(u,\mu,\nu)\mapsto M^\alpha_\ve(u)+2d_\alpha(f^+,\mu)+2d_\alpha(f^-,\mu).
\end{equation}
It is an easy consequence of our previous $\Gamma-$convergence result and of the fact that the quantity $G((\mu,\nu),(f^+,f^-)):=2d_\alpha(f^+,\mu)+2d_\alpha(f^-,\mu)$ is continuous on this space that we still have $\Gamma-$convergence to the functional $M^\alpha(u)+G((\mu,\nu),(f^+,f^-))$. The minimization of this last functional being equivalent to the minimization of $M^\alpha(u)$ under the divergence constraint $\nabla\cdot u=\mu-\nu$ one has obtained a useful approximation of any branched transport problem.

Notice that in this framework the imprvement that we proposed at the end of Section 5 is crucial: without this kind of correction on the divergence, one would got approximating measures $u_\ve$ (in the $\Gamma-\limsup$ construction) whose divergence is no more the difference of two probabilities, since some extra mass at every node is added both to the positive and to the negative part. 

Nevertheless, the approximation \label{functionals with pen} is not satisfactory yet, especially for numerical purposes: we are trying to suggest methods to approximate the minimization of $M^\alpha$ and we propose to use the distance $d_\alpha$ itself!! For computing it one should probably solve a problem of the same kind and no progress would have been done. 

A possible escape we suggest is replacing $d_\alpha$ with other quantities, which are larger but still vanish when the two measures coincide. This is possible for instance thanks to \cite{MorSan}, where the inequality $d_\alpha\leq C W_1^{2\alpha-1}$ is proven, $W_1$ being the usual Wasserstein distance on $\pical(\Ov)$. We will call $G_1$ the functional $G_1((\mu,\nu),(f^+,f^-)):=CW_1^{2\alpha-1}(\mu,f^+)+CW_1^{2\alpha-1}(\nu,f^-)$. The distance $W_1$ has the advantages of being independent of the branched transport problem and of admitting more numerical methods for computing it (it is, by the way, a dual Lipschitz norm). Notice that other possible attachment terms, easier to compute, such as the $L^2$ norm, are not adequate because they are not continuous with respect to the convergence we have (and because in general we do not want to require $f\in L^2$).

It is anyway important to stress that, the approach on $\MO$ without penalization stays useful for a lot of problems where the divergence is not prescribed but enters the optimization (think at $\min_\mu d_\alpha(\mu,\nu)+F(\mu)$). Some of this problems are addressed in \cite{landscape}, for instance for urban planning or biological shape optimization. 

Yet, even if there are, as we showed, was to overcome the problem, we find anyway interesting to ask the following question:

\begin{question}
Given $f$, is it possible to find a suitable sequence $f_\ve\deb f$ so that one can prove $\Gamma-$convergence of the functionals $u\mapsto M^\alpha_\ve(u)+I_{\nabla\cdot u = f_\ve}$ to $u\mapsto M^\alpha(u)+I_{\nabla\cdot u = f}$ (being $I$ the indicator function in the convex analysis sense, i.e. $+\infty$ if the condition is not verified or zero if it is)? is it possible to find $f_\ve$ explicitly, for instance as a convolution of $f$ with a given kernel?
\end{question}

The second issue we want to address, after the one concerning divergence constraints, deals with the convergence of the minimizers. $\Gamma-$convergence is quite useless if we cannot deduce that the minimizers $u_\ve$ converge, at least up to subsequences, to a minimizer $u$. Yet, this requires a little bit of compactness. The compactness we need is compactness in $\MO$, i.e. we want bounds on the mass of $\nabla\cdot u_\ve$ and of $u_\ve$. The first bound, has been guaranteed by the fact that we decided to stick to the case of difference of probability measures. On the contrary, the bound on $|u_\ve|(\Omega)$ has to be proven.

Notice that $M^\alpha_\ve(u_\ve)\leq C$ will not be sufficient for such a bound, as one can guess looking at the limit functional. Think at a finite graph with a circle of length $l$ and mass $m$ on it: its energy is $m^\alpha l$ which provides no bound on $ml$ (its mass), if $m$ is allowed to be large. Actually, what happens on the limit functional is ``bounded energy configuration have not necessarily bounded mass, but optimal configuration do have''. This is due to the fact that, if $f^+$ and $f^-$ are probabilities, then $m\leq 1$ on optimal configurations (and no cycle are possible, by the way). Notice that this statement does not depend on $m\mapsto m^\alpha$ being concave, but simply increasing in $m$. The same kind of behavior is likely to be true on the approximating problems, but after months of work a proof of this fact did not appear. 

Hence we conclude with two points: we ask a question and we give a final suggestion for getting a useful $\Gamma$-convergence result.

The suggestion is quite {\it naïve} : just take a sufficiently large number $K$ so that every minimizer $u$ of the limit problem satisfies $|u|(\Ov)\leq K$ and then restrict the analysis to the compact subset $Y_K(\Omega):=\{(u,\mu,\nu)\in Y(\Omega)\;:\;|u|(\Ov)\leq K\}$. Unfortunately, there exists no continuous functional $\Psi$ on $Y(\Omega)$ which is coercive on the mass of $u$. If it was the case, one could have added a term f the form $F(u)=\left(\Psi(u)- K\right)_+^2$ to the limit and approximating functionals. It is on the contrary a reasonable question, simpler than the one we are going to detail below, whether there exists a continuous functional $\Psi$ such that, for the minimizers $u_\ve$, the condition $\Psi(u_\ve)\leq C$ implies a bound on the mass (taking advantage of possible extra estimates on $u_\ve$). 

Thus, the problems 
$$\min\left\{ M^\alpha_\ve(u)+G_1((\mu,\nu)(f^+,f^-))\;(u,\mu,\nu)\in Y_K(\Omega)\right\}$$
approximate (with $\Gamma-$convergence) the limit problem given by $\min_{Y_K(\Omega)}\; M^\alpha(u)+G_1((\mu,\nu),(f^+,f^-)) $, which is equivalent to $\min\; M^\alpha(u)\,:\,\nabla\cdot u = f$.

Here as well a proper use of the estimate given in the last improvement in Section 5 are important. Actually, one needs to produce a recovery sequence $u_\ve$ for the $\Gamma-\limsup$ construction which does not violate the mass constraint given by $K$. This may be obtained in the following way: instead of considering as a dense subset the set of all vector measures $u$ corresponding to finite graphs, take the set of all such measures $u$ satisfying $|u|<K$. This set is still dense in energy. Afterwards, one notices that the construction which was provided satisfied $|u_\ve|(\Ov)\leq |u|(\Ov)+C\ve^{\alpha+1}$ (where the constant $C$ depends actually on $u$, but not on $\ve$, since it depends on the number of nodes and of branches at every node). This means that, if $u$ satisfies the strict inequality $|u|<K$, this will be true for $u_\ve$ as well for $\ve$ sufficiently small. 

And, as obvious, the question is
\begin{question}
Prove a bound on the $L^1$ norm of the minimizers $u_\ve$ (or on suitable minimizers $u_\ve$, if needed). If possible, prove it for minimizers $u_\ve$ which minimize $M^\alpha_\ve$ under a divergence constraint $\nabla \cdot u_\ve = f_\ve$ (so that it will be true even if we add penalizations on the divergence). Is is possible to do it simply by writing and exploiting optimality conditions under the form of elliptic (second or fourth order) PDEs? do we need to pass through more geometrical tools such has some surrogate of the no-cycle condition? Is this bound equivalent to a bound on some $m^\alpha$ (being $m$ the flux thrugh some segments or curves, as in what we proved in Step 1, where $m(x)$ is the flux through $R_x$) coupled with some bound on the fluxes $m$?
\end{question}

For finishing this questioning section, here is the last natural one:
\begin{question}
Prove the same results as in this paper or investigate what happens in $\R^d$, for $d\geq 3$.
\end{question}

Notice that our proofs almost never used (up to the last improvement in Section 5) any gradient decomposition of vector fields in $\R^2$ such as $u=\nabla\phi+Rot\nabla\psi$. Yet, this would have been a typical trick for managing divergences in two dimensions. In our opinion the point where we used the most the fact that we are in $\R^2$ is when we disintegrate with respect to $x$ and we estimate the total variation in $y$ by the oscillation (which is very one-dimensional).

\section*{Acknowledgements}
\noindent
Jean-Michel Morel is warmly thanked for the impulse he gave to this subject at his very starting point in 2006 and for the constant help and attention towards both the topic and the author since then.

Giovanni Alberti and Giuseppe Buttazzo contributed to the current version of the paper by means of the interesting remarks and clarifying comments they proposed when this work was first presented in a seminar in Pisa. 

The author also took advantage of the support of the ANR project OTARIE (Optimal transport : Theory and Applications
to cosmological Reconstruction and Image processing) and of the French-Italian Galilée Project ``Modèles de réseaux de transport:
phénomènes de congestion et de branchement''

\end{document}